\documentclass{article}
\usepackage{amsmath}
\usepackage{amsthm}
\usepackage{verbatim}
\usepackage{amssymb}
\usepackage{delarray}

\usepackage{epsfig}
\usepackage{color}
\usepackage{mathrsfs}
\usepackage{pgfplots}
\usepackage{fancyhdr}
\pgfplotsset{compat=newest}

\usepgfplotslibrary{clickable}

\newtheorem{lemma}{Lemma}
\newtheorem{thm}{Theorem}
\newtheorem{defn}{Definition}
\newtheorem{conj}{Conjecture}
\newtheorem{rem}{Remark}
\newtheorem{cor}{Corollary}

\title{Arithmetic properties of generalized Fibonacci sequences}
\author{Soohyun Park\\ Massachusetts Institute of Technology\\ Cambridge, MA 02139 \\
\texttt{soopark@mit.edu}}

\date{08/08/13}
\begin{document}

\maketitle

\begin{abstract}
The $\emph{generalized}$ Fibonacci sequences are sequences $\{f_n\}$ which satisfy the recurrence $f_n(s, t) = sf_{n - 1}(s, t) + tf_{n - 2}(s, t)$ ($s, t \in \mathbb{Z}$) with initial conditions $f_0(s, t) = 0$ and $f_1(s, t) = 1$. In a recent paper, Amdeberhan, Chen, Moll, and Sagan considered some arithmetic properites of the generalized Fibonacci sequence. Specifically, they considered the behavior of analogues of the $p$-adic valuation and the Riemann zeta function. In this paper, we resolve some conjectures which they raised relating to these topics. We also consider the rank modulo $n$ in more depth and find an interpretation of the rank in terms of the order of an element in the multiplicative group of a finite field when $n$ is an odd prime. Finally, we study the distribution of the rank over different values of $s$ when $t = -1$ and suggest directions for further study involving the rank modulo prime powers of generalized Fibonacci sequences. 
\end{abstract}

\section{Introduction}
Recall that the Fibonacci sequence $\{F_n\}$ is defined as the sequence satisfying the recurrence $F_n = F_{n - 1} + F_{n - 2}$ with $F_0 = 0$ and $F_1 = 1$. The terms of this sequence satisfy some interesting divisibility properties, which we list below.

\begin{thm}[Vorob'ev]\cite{V}
gcd($F_n, F_m$) = $F_{\text{gcd}(n, m)}$ for all $m, n \in \mathbb{Z^{+}}$. Note that this implies that gcd($F_n$, $F_{n + 1}$) = 1. 
\end{thm}

\begin{thm}\cite{V}
If $n \ge 2$, $F_n | F_m$ if and only if $n | m$.
\end{thm}

\begin{thm}[Hoggatt and Long]\cite{HL}
$F_n$ is prime only if $n$ is prime.
\end{thm}

\begin{thm}\cite{HL}
Let $r \in \mathbb{Z^{+}}$. Then there exists an $m$ such that $r | F_m$. If $m$ is the least such number, then $r | F_n$ if and only if $m | n$. 
\end{thm}

Many of these divisibility properties are also shared in the following generalization of the Fibonacci sequence.

\begin{defn}[ACMS]\cite{ACMS}
The $\emph{generalized}$ Fibonacci polynomials are polynomials in $s$ and $t$ defined by $f_0(s, t) = 0$ and $f_1(s, t) = 1$ with the recurrence $f_n(s, t) = sf_{n - 1}(s, t) + tf_{n - 2}(s, t)$ for $n \ge 2$. 
\end{defn}

Here are some counterparts of the results above for generalized Fibonacci polynomials. Note that divisibility is considered over $\mathbb{Z}[s, t]$ unless otherwise indicated.

\begin{thm}\cite{HL}
gcd($f_n, f_m$) = $f_{\text{gcd}(n, m)}$ for all $m, n \in \mathbb{Z^{+}}$. Note that this implies that gcd($f_n$, $f_{n + 1}$) = 1. 
\end{thm}

\begin{thm}\cite{HL}
If $n \ge 2$, $f_n | f_m$ if and only if $n | m$.
\end{thm}

\begin{thm}\cite{HL}
$f_n(s, t)$ is irreducible over $\mathbb{Q}$ if and only if $n$ is prime.
\end{thm}

\begin{thm}\cite{HL}
Fix $s, t \in \mathbb{Z}$ and let $r$ be a positive integer with gcd($r$, $t$) = 1. Then there exists $m$ such that $r | f_m$. If there is a  least positive integer $m$ such that $r | f_m$, then $r | f_n$ if and only if $m | n$. 
\end{thm}

The remainder of this paper will mainly deal with generalized Fibonacci sequences, which are the sequences obtained after fixing $s, t \in \mathbb{Z}$ in the generalized Fibonacci polynomials. In the next section, we will consider the periodicity of generalized Fibonacci sequences modulo $n$ in relation to a generalization of $p$-adic valuations. In Section 3, we consider an analogue of the Riemann zeta function for generalized Fibonacci sequences. Finally, we will examine the rank modulo $n$ in more depth and find an interpretation as the order of an element in a finite field.

\section{$d$-adic valuations of generalized Fibonacci sequences}
One property of the Fibonacci sequence which carries over to generalized Fibonacci sequences is periodicity modulo $n$ for any $n \ge 2$. This property also holds holds for generalized Fibonacci sequences since there are finitely many choices for pairs of consecutive terms modulo $n$. According to \cite{War}, many arithmetic properties of generalized Fibonacci sequences ultimately depend on this periodicity property and the divisibility property which was shown in Theorem 2 and Theorem 6. We now formally define the period modulo $m$.

\begin{defn}
The $\emph{period}$ modulo $m$ is defined to be the smallest positive integer $k(m)$ such that $f_n(s, t) \equiv f_{n + k(m)}(s, t)$ (mod $m$) for all sufficiently large $n$. 
\end{defn} 

There is a quantity similar to the period related to divisibility modulo $m$.

\begin{defn}
If there exists $r(m)$ such that $f_y(s,t) \equiv 0$ (mod $m$) if and only if $r(m) | y$, $r(m)$ is defined as the $\emph{rank}$ modulo $m$. 
\end{defn}

For example, we can consider the Fibonacci sequence modulo 8. Here are the first few terms of the residues modulo 8: $0, 1, 1, 2, 3, 5, 0, 5, 5, 2, 7, 1, 0, 1, 1, \ldots$ In this case, the period $k(8) = 12$, while the rank $r(8) = 6$ \cite{Ro}. \\

We now show that the rank exists modulo a prime $p$ if and only if $p \nmid t$.

\begin{thm}
Let $p$ be a prime. Then, $r(p)$ exists if and only if $p \nmid t$.
\end{thm}

\begin{proof}
By Theorem 8, there is a positive integer $m$ such that $p | f_m$ if $p \nmid t$. This means that it suffices to show that it is not possible to obtain a sequence of the form $x, 0, 0, \ldots$ modulo $p$ with $x \ne 0$ in this case. Using the recurrence, we find that this is only possible if there is a $t$ such that $tx \equiv 0$ (mod $p$). If $p | s$ and $p | t$, then this is true for any $x$ and $f_n \equiv 0$ (mod $p$) for all $n \ge m$, so $r(p)$ does not exist. If $p \nmid s$ and $p | t$, the sequence would be of the form $0, 1, s, s^2, s^3, \ldots$ mod $p$ and there would be no positive integer $m$ such that $p | f_m$. If $p \nmid t$, this statement implies that $x \equiv 0$ (mod $p$). However, this contradicts the assumption that $x \not\equiv 0$ (mod $p$). So, the pair $(f_m, f_{m + 1})$ must be of the form $(0, a)$ modulo $p$, where $a \ne 0$ and the sequence up to $f_{m - 1}$ is multiplied by $a$ mod $p$. Then, $r(p)$ is equal to the smallest possible value of $m$.
\end{proof}

Arguments similar to the one used above can be used to show that $r(p^k)$ exists if and only if $r(p)$ exists as $p^k$ is relatively prime to $t$ if and only if $p \nmid t$. We will mainly focus on the case where $t = -1$ since $r(m)$ exists for any $m \ge 2$ and many properties of the original Fibonacci sequence carry over to this case. \\

Here are some results on the rank of generalized Fibonacci sequences modulo $p$, where $p$ is an odd prime.

\begin{thm}[Li] \cite{L}
Let $p$ be an odd prime and $r(p)$ be the rank of the generalized Fibonacci sequence with parameters $(s, t)$ mod $p$. Let $D = s^2 + 4t$. If $p | D$, we have that $r(p) = p$.
\end{thm}

\begin{thm} \cite{L}
If $p$ is an odd prime such that $p \nmid D$, then $(-t/p) = 1$ if and only if $r(p) | \frac{p - (D/p)}{2}$. 
\end{thm}

\begin{rem}
Note that we always have $(-t/p) = 1$ when $t = -1$. By Theorem 10 and Theorem 11, we have that $r(p) \le p$ for all odd primes $p$. 
\end{rem}

Another property of the Fibonacci sequence which carries over to generalized Fibonacci sequences with $t = -1$ is that we can relate periods of generalized Fibonacci sequences mod $p$ to $p$-adic valuations of their elements. This is one example of the relationship of periodicity modulo $n$ to other arithmetic properties of generalized Fibonacci sequences. Here are some results which have similar counterparts for the Fibonacci sequence (see \cite{Le}, \cite{Wl}, \cite{K}, and \cite{Le} for Lemmas 1, 2, 3, and 4 respectively). We use $g_n(s,t)$ to denote the generalized Lucas sequences.

\begin{lemma}
$\text{gcd}(f_n(s, -1), g_n(s, -1)) \le 2$.
\end{lemma}

\begin{proof}
\begin{align*}
\text{gcd}(f_n(s, -1), g_n(s, -1)) &= \text{gcd}(f_n(s, -1), f_{n + 1}(s, -1) - f_{n - 1}(s, -1)) \text{ (Proposition 2.3 of \cite{ACMS})} \\
&= \text{gcd}(f_n(s, -1), sf_n(s, -1) - 2f_{n - 1}(s, -1)) \\
&= \text{gcd}(f_n(s, -1), - 2f_{n - 1}(s, -1)) \\
&= \text{gcd}(f_n(s, -1), 2) \text{ (follows by induction)}
\end{align*}
\end{proof}

\begin{lemma}
$f_{kn}(s, -1) = 2^{1 - k} f_n(s, -1)(K f_n(s, -1)^2 + k g_n(s, -1)^{k - 1})$, where $K \in \mathbb{Z}$.
\end{lemma}

\begin{proof}
Consider the binomial expansion of $f_{kn}(s, -1) = \frac{X^{kn} - Y^{kn}}{X - Y}$, where $X = \frac{s + \sqrt{s^2 - 4}}{2}$ and $Y = \frac{s - \sqrt{s^2 - 4}}{2}$.

\begin{align*}
f_{kn}(s, -1) &= \frac{X^{kn} - Y^{kn}}{X - Y} \\
&= \frac{1}{X - Y} (\frac{1}{2^k}((X - Y)f_n(s, -1) + g_n(s, -1))^k - \frac{1}{2^k}(-(X - Y)f_n(s, -1) + g_n(s, -1))^k) \\
&= 2^{1 - k} \sum_{j \text{ odd}} \binom{k}{j} (X - Y)^{j - 1} f_n(s, -1)^j g_n(s, -1)^{k - j} \\
&= 2^{1 - k} f_n(s, -1) (K f_n(s, -1)^2 + k g_n(s, -1)^{k - 1}) \text{ } (\text{where $K \in \mathbb{Z}$})
\end{align*}

\end{proof}

\begin{lemma}
Let $k \in \mathbb{N}$. We have $f_{kn}(s, -1) \equiv kf_{n + 1}(s, -1)^{k - 1} f_n(s, -1)$ (mod $f_n(s, -1)^2$) and $f_{kn + 1}(s, -1) \equiv f_{n + 1}(s, -1)^k$  (mod $f_n(s, -1)^2$).
\end{lemma}

\begin{proof}
We use induction on $k$. If $k = 1$, we have $f_n(s, -1) \equiv f_n(s, -1)$ (mod $f_n(s, -1)^2$) and $f_{n + 1}(s, -1) \equiv f_{n + 1}(s, -1)$ (mod $f_n(s, -1)^2$). Assume that $f_{kn}(s, -1) \equiv kf_{n + 1}(s, -1)^{k - 1} f_n(s, -1)$ (mod $f_n(s, -1)^2$) and $f_{kn + 1}(s, -1) \equiv f_{n + 1}(s, -1)^k$  (mod $f_n(s, -1)^2$). Then, we have the following.

\begin{align*}
f_{(k + 1)n}(s, -1) &= f_{kn + 1}(s, -1)f_n(s, -1) - f_{kn}(s, -1)f_{n - 1}(s, -1) \text{ (Theorem 2.2 of \cite{ACMS})} \\
&\equiv f_{n + 1}(s, -1)^k f_n(s, -1) - kf_{n + 1}(s, -1)^{k - 1} f_n(s, -1) f_{n - 1}(s, -1) \text{ (mod $f_n(s, -1)^2$)} \\
&\equiv f_{n + 1}(s, -1)^k f_n(s, -1) + kf_{n + 1}(s, -1)^{k - 1} f_n(s, -1) f_{n + 1}(s, -1) \text{ (mod $f_n(s, -1)^2$)} \\
&\equiv (k + 1)f_{n + 1}(s, -1)^k f_n(s, -1) \text{ (mod $f_n(s, -1)^2$)} \\
\\
f_{(k + 1)n + 1}(s, -1) &= f_{kn + 1}(s, -1)f_{n + 1}(s, -1) - f_{kn}(s, -1)f_n(s, -1) \\
&\equiv f_{n + 1}(s, -1)^k f_{n + 1}(s, -1) - kf_{n + 1}(s, -1)^{k - 1} f_n(s, -1)^2 \text{ (mod $f_n(s, -1)^2$)} \\
&\equiv f_{n + 1}(s, -1)^{k + 1} \text{ (mod $f_n(s, -1)^2$)} 
\end{align*}

\end{proof}

\begin{lemma}
Let $p$ be an odd prime and $e = e(p) = \nu_p(f_{r(p)}(s, -1))$. Then,
$$
\nu_p(f_n(s, -1)) = 
\begin{cases}
\nu_p(n) - \nu_p(r(p)) + e(p), & \text{if $n \equiv 0$ (mod $r(p)$).} \\
0 , & \text{if $n \not\equiv 0$ (mod $r(p)$).}
\end{cases}
$$
\end{lemma}

\begin{proof}
By Lemma 2, we have that $f_{kn}(s, -1) = 2^{1 - k} f_n(s, -1)(K f_n(s, -1)^2 + k g_n(s, -1)^{k - 1})$, where $K \in \mathbb{Z}$. Let $k  = p$, $n = c r(p) p^{\alpha - 1}$, where gcd($c, p$) = 1 and $\alpha \ge 1$. This means that $f_{cr(p)p^\alpha}(s, -1) = 2^{1 - p} f_{cr(p)p^{\alpha - 1}}(s, -1) (K'p^2 + pg_{cr(p)p^{\alpha - 1}}(s, -1)^{p - 1})$ ($K' \in \mathbb{Z}$). Note that $p \nmid g_{cr(p)p^{\alpha - 1}}(s, -1)$ since \\ $\text{gcd}(f_n(s, -1), g_n(s, -1)) \le 2$ by Lemma 1 and $p$ is an odd prime. This implies that $\nu_p(f_{cr(p)p^\alpha}(s, -1)) = \nu_p(f_{cr(p)p^{\alpha - 1}}(s, -1)) + 1$ and it follows by induction that $\nu_p(f_{cr(p)p^\alpha}(s, -1)) = \nu_p(f_{cr(p)}) + \alpha$. \\

We also have by Lemma 3 that $f_{cr(p)}(s, -1) \equiv cf_{r(p)}(s, -1)f_{r(p) + 1}(s, -1)^{c - 1}$ (mod $p^{2e}$). Note that gcd$(f_{r(p) + 1}(s, -1), p) = 1$ since $f_m(s, -1) \equiv 0$ (mod $p$) if and only if $r(p) | m$. Since  $r(p) \nmid r(p) + 1$, we have $p \nmid f_{r(p) + 1}(s, -1) \Rightarrow $ gcd$(f_{r(p) + 1}(s, -1), p) = 1$ since $p$ is prime. This implies that $\nu_p(f_{cr(p)}(s, -1)) = \nu_p(f_{r(p)}(s, -1))$ when gcd($c$, $p$) = 1. This can be used with  $\nu_p(f_{r(p)}(s, -1)) = e < 2e$ to show that $\nu_p(f_{cr(p)p^\alpha}(s, -1)) = \nu_p(f_{r(p)}(s, -1)) + \alpha = e(p) + \alpha$. Thus, $\nu_p(f_n(s, -1)) = \nu_p(n) - \nu_p(r(p)) + e(p)$ when $r(p) | n$.

\end{proof}

\begin{rem}
It would be interesting to generalize this lemma to relate $p$-adic valuations to the rank modulo $p$ for other values of $t$.
\end{rem}

The results above relating the period of the generalized Fibonacci sequence to the $p$-adic valuations of its elements can be used to show that the following conjecture of Amdeberhan, Chen, Moll, and Sagan holds.

\begin{conj}\cite{ACMS}
Suppose $s \ge 2$ is an integer and $d \ge 3$ is an odd integer. There exist integers $s'(s, d)$ and $d'(s, d)$ such that $d' \le d$ and $\nu_d(f_n(s, -1)) = \delta_{d' \mathbb{Z}}(n)\nu_d(\frac{s'n}{d'})$.
\end{conj}

We will first show that this conjecture holds when $d$ is an odd prime.

\begin{thm}
Suppose $s \ge 2$ is an integer and $p \ge 3$ is an odd prime. There exist integers $s'(s, d), d'(s, d)$ such that $d' \le d$ and $\nu_p(f_n(s, -1)) = \delta_{d' \mathbb{Z}}(n)\nu_p(\frac{s'n}{d'})$.
\end{thm}

\begin{proof}
If there exists $d'$ such that if $\delta_{d' \mathbb{Z}}(n) = 0$, we have that $\nu_d(f_n(s, -1)) = 0$. This implies that $d' \nmid n \Rightarrow d \nmid f_n(s, -1)$, which is equivalent to the statement $d | f_n(s, -1) \Rightarrow d' | n$. Note that $d'$ divides the rank mod $p$. Let $r(p)$ be the rank mod $p$. Setting $d' = r(p)$ satisfies the previous condition and we have $d' \le p$ (see remark following Theorem 4). \\

Consider the case where $r(p) | n$. Then, $\nu_p(f_n(s, -1)) = \nu_p(\frac{s'n}{d'})$ if and only if $\nu_p(n) - \nu_p(r(p)) + e(p) = \nu_p(\frac{s'n}{d'})$. If we take $s'$ such that $d' | s'$, we have $\nu_p(\frac{s'n}{d'}) = \nu_p(n) + \nu_p(\frac{s'}{d'})$ and the statement reduces to $\nu_p(\frac{s'}{d'}) = -\nu_p(r(p)) + e(p)$. Note that we either have $\nu_p(r(p))$ equal to 0 or 1, which means that $-\nu_p(r(p)) + e(p) \ge 0$ as $p | f_{r(p)}(s, -1)$. This equation is satisfied if we set $s' = d' \cdot p^{- \nu_p(r(p)) + e(p)}$, which only depends on $s$ and $p$. 

\end{proof}

This result can be generalized to the case where $d$ is the power of an odd prime. 

\begin{thm}
Suppose $s \ge 2$ is an integer and $d = p^r$, where $p$ is an odd prime ($p \ge 3$). There exist integers $s'(s, d)$ and $d'(s, d)$ such that $d' \le d$ and $\nu_d(f_n(s, -1)) = \delta_{d' \mathbb{Z}}(n)\nu_d(\frac{s'n}{d'})$.
\end{thm}

\begin{proof}
We set $d'$ equal to some divisor of $r(d)$ less than or equal to $d$ (eg. $d' = 1$). We now consider the case where $r(d) | n$. We claim that there exists an $s'$ divisible by $d'$ such that $\nu_d(f_n(s, -1)) = \nu_d(\frac{s'n}{d'})$. Since $\nu_{p^r}(N) = \lfloor \frac{\nu_p(N)}{r} \rfloor$, we want to find $s'$ such that $\lfloor \frac{\nu_p(n) - \nu_p(r(p)) + e(p)}{r} \rfloor = \lfloor \frac{\nu_p(\frac{s'n}{d'})}{r} \rfloor = \lfloor \frac{-\nu_p(n) + \nu_p(\frac{s'}{d'})}{r} \rfloor$. We find that setting $s' = d' \cdot p^{-\nu_p(r(p)) + e(p)}$ as in the end of the proof of Theorem 12 also satisfies this equation since the numerators are equal.

\end{proof}

We can use solutions for powers of odd primes to show that the conjecture holds for arbitrary odd integers $d$. 

\begin{thm}
Suppose $s \ge 2$ is an integer and $d \ge 3$ is an odd integer. There exist integers $s'(s, d)$ and $d'(s, d)$ such that $d' \le d$ and $\nu_d(f_n(s, -1)) = \delta_{d' \mathbb{Z}}(n)\nu_d(\frac{s'n}{d'})$.
\end{thm}

\begin{proof}
Set $d'$ to be a divisor of the rank mod $d$ less than or equal to $d$ as above. Let $d = p_1^{\alpha_1} \cdots p_k^{\alpha_k}$ and $s' = d' \cdot \displaystyle\prod_{j = 1}^k {p_j^{-\nu_{p_j}(r(p_i)) + e(p_j)}}$. We have $\nu_{p_i}(n) = \nu_{p_i}(n) - \nu_{p_i}(r(p_i)) + e(p_i)$ and $\nu_{p_i}(\frac{s'n}{d'}) =\nu_{p_i}(n) + \nu_{p_i}(\frac{s'}{d'}) = \nu_{p_i}(n) + \nu_{p_i}(\displaystyle\prod{p_j^{\nu_{p_j}(-r(p_j)) + e(p_j)}}) = \nu_{p_i}(n) - \nu_{p_i}(r(p_i)) + e(p_i)$. So, we have $\nu_{p_i}(f_n(s, -1)) = \nu_{p_i}(\frac{s'n}{d'})$ for all $i$. Since $\nu_d(N) = \min_{1 \le i \le k} \lfloor \frac{\nu_{p_i}(N)}{\alpha_i} \rfloor$ for $N \in \mathbb{N}$, we have $\nu_d(f_n(s, -1)) = \min_{1 \le i \le k} \lfloor \frac{\nu_{p_i}(f_n(s, -1))}{\alpha_i} \rfloor = \min_{1 \le i \le k} \lfloor \frac{\nu_{p_i}(\frac{s'n}{d'})}{\alpha_i} \rfloor = \nu_d(\frac{s'n}{d'})$.

\end{proof}

\section{An analogue of the Riemann zeta function}
We now consider an analogue of the Riemann zeta function. The Riemann zeta function is defined as the analytic continuation of $\zeta(z)  = \sum_{k = 1}^\infty \frac{1}{k^z}$, where $z \in \mathbb{C}$. One variation of this function which has been considered is the function $\zeta_F(z)  = \sum_{k = 0}^\infty \frac{1}{F_k^z}$ with $z \in \mathbb{C}$, where the positive integers are replaced with terms of the Fibonacci sequence \cite{WZ}. This function shares some properties with the original Riemann zeta function (see \cite{M} for more details) and has been studied in several different ways (see \cite{HK} for an overview). Some of the work that has been done with this analogue of the Riemann zeta function involves estimating the tails of the series for positive integers $z$. The first result relating to such estimates is the following.

\begin{thm} [Ohtsuka and Nakamura] \cite{ON} Let $\{F_k\}$ be the Fibonacci sequence.
$$
\left\lfloor \left(\sum_{k = n}^\infty \frac{1}{F_k}\right)^{-1} \right\rfloor = 
\begin{cases}
F_{n - 2}, & \text{if $n$ even, $n \ge 2$} \\
F_{n - 2} - 1, & \text{if $n$ odd, $n \ge 1$}
\end{cases}
$$  

$$
\left\lfloor \left(\sum_{k = n}^\infty \frac{1}{F_k^2}\right)^{-1} \right\rfloor = 
\begin{cases}
F_{n - 1}F_n - 1, & \text{if $n$ even, $n \ge 2$} \\
F_{n - 1}F_n, & \text{if $n$ odd, $n \ge 1$}
\end{cases}
$$  
\end{thm}

Note that no closed form is known for the sum $\sum_{k = n}^\infty \frac{1}{F_k}$ although some of its properties are known. Holliday and Komatsu \cite{HK} gave the first result relating to sums of more general terms. Specifically, they studied the case where the denominators were Fibonacci polynomials with a fixed integer parameter. Here is one of their results.

\begin{thm}[Holliday and Komatsu] \cite{HK}
If $t = 1$ and $s, n \in \mathbb{Z}^{+}$, then

\begin{align*}
\left\lfloor \left(\sum_{k = n}^\infty \frac{1}{f_{k}(s, 1)}\right)^{-1} \right\rfloor = f_{n}(s, 1) - f_{n - 1}(s, 1) - \delta_{\mathbb{Z} \setminus 2\mathbb{Z}}(n),
\end{align*}

and

\begin{align*}
\left\lfloor \left(\sum_{k = n}^\infty \frac{1}{f_{k}(s, 1)^2}\right)^{-1} \right\rfloor = s f_n(s, 1) f_{n - 1}(s, 1) - \delta_{2\mathbb{Z}}(n).
\end{align*}

\end{thm}

According to \cite{ACMS}, Holliday and Komatsu asked whether the above result could be generalized for other $t$. The following result from \cite{ACMS} generalizes the first sum and also considers a more general class of sums than those considered in \cite{HK}. 

\begin{thm} [ACMS] \cite{ACMS}
If $s \ge t \ge 1$ and $n, r \in \mathbb{Z}^+$ then 

\begin{align*}
\left\lfloor \left(\sum_{k = n}^\infty \frac{1}{f_{rk}(s, t)}\right)^{-1} \right\rfloor = f_{rn}(s, t) - f_{r(n - 1)}(s, t) - \delta_{2\mathbb{Z}}(r(n - 1)).
\end{align*}

If $t = 1$ and $s, n, r \in \mathbb{Z}^+$ then

\begin{align*}
\left\lfloor \left(\sum_{k = n}^\infty \frac{1}{f_{rk}(s, 1)^2}\right)^{-1} \right\rfloor = f_{rn}(s, 1)^2 - f_{r(n - 1)}(s, 1)^2 - \delta_{2\mathbb{Z}}(r(n - 1)).
\end{align*}

\end{thm}

It is conjectured that there is an analogue of this theorem replacing $t$ with $-t$.

\begin{conj}[ACMS] \cite{ACMS}
If $s > t \ge 1$ with $(s, -t) \ne (2, -1)$ and $n, r \in \mathbb{Z}^+$ then 

\begin{align*}
\left\lfloor \left(\sum_{k = n}^\infty \frac{1}{f_{rk}(s, -t)}\right)^{-1} \right\rfloor = f_{rn}(s, -t) - f_{r(n - 1)}(s, -t) - 1.
\end{align*}

If $t = -1$ and $s, n, r \in \mathbb{Z}^+$ then

\begin{align*}
\left\lfloor \left(\sum_{k = n}^\infty \frac{1}{f_{rk}(s, -1)^2}\right)^{-1} \right\rfloor = f_{rn}(s, -1)^2 - f_{r(n - 1)}(s, -1)^2 - 1.
\end{align*}

\end{conj}

We show that the conjecture holds for sufficiently large $s, t$.

\begin{thm}
Let $s > t \ge 1$ be a pair of integers with $(s, -t) \ne (2, -1)$ and $n, r \in \mathbb{Z}^+$. If $s, t$ are sufficiently large, we have that 

\begin{align*}
\left\lfloor \left(\sum_{k = n}^\infty \frac{1}{f_{rk}(s, -t)}\right)^{-1} \right\rfloor = f_{rn}(s, -t) - f_{r(n - 1)}(s, -t) - 1.
\end{align*}

Let $t = -1$ and $s, n, r \in \mathbb{Z}^+$. If $s$ is sufficiently large, we have that

\begin{align*}
\left\lfloor \left(\sum_{k = n}^\infty \frac{1}{f_{rk}(s, -1)^2}\right)^{-1} \right\rfloor = f_{rn}(s, -1)^2 - f_{r(n - 1)}(s, -1)^2 - 1.
\end{align*}

\end{thm}

We will use the following result in \cite{ACMS} in the proof of this theorem.

\begin{lemma} [ACMS] \cite{ACMS}
Let $r, m, n \in  \mathbb{P}$ and $s, t$ be arbitrary integers. We have $f_{rn}(s, t) f_{r(n + m - 1)}(s, t) - f_{r(n - 1)}(s, t) f_{r(n + m)}(s, t) = (-t)^{r(n - 1)} f_r(s, t) f_{rm}(s, t)$.
\end{lemma}

Now we begin the proof of Theorem 18.

\begin{proof}

We will first show that each of the two series $\sum_{k = n}^\infty \frac{1}{f_{rk}(s, -t)}$ and $\sum_{k = n}^\infty \frac{1}{f_{rk}(s, -1)^2}$ converges under the given conditions. By Proposition 1.1 of \cite{ACMS}, we have that$f_n(s, -t) = \frac{X^n - Y^n}{X - Y}$, where $X = \frac{s + \sqrt{s^2 - 4t}}{2}$ and $Y = \frac{s - \sqrt{s^2 - 4t}}{2}$. This means that the ratio of consecutive terms of the sum $\sum_{k = n}^\infty \frac{1}{f_{rk}(s, -t)}$ is of the form $\frac{X^p - Y^p}{X^{p + r} - Y^{p + r}}$ for some $p \ge rn$. Similarly, we have that the ratio of consecutive terms of the sum $\sum_{k = n}^\infty \frac{1}{f_{rk}(s, -1)^2}$ is of the form $\left(\frac{X^p - Y^p}{X^{p + r} - Y^{p + r}}\right)^2$ for some $p \ge rn$. Note that $\lim_{p \to \infty} \frac{X^p - Y^p}{X^{p + r} - Y^{p + r}} = \frac{1}{X^r}$. Since $X > 1$ for all pairs $(s, -t) \ne (2, -1)$, it follows that the two series $\sum_{k = n}^\infty \frac{1}{f_{rk}(s, -t)}$ and $\sum_{k = n}^\infty \frac{1}{f_{rk}(s, -1)^2}$ both converge. \\

We now examine the sum $\sum_{k = n}^\infty \frac{1}{f_{rk}(s, -t)}$. Let $B(n) = \sum_{k = n}^\infty \frac{1}{f_{rk}(s, -t)}$. Note that \\ $\left\lfloor \left(\sum_{k = n}^\infty \frac{1}{f_{rk}(s, -t)}\right)^{-1} \right\rfloor = f_{rn}(s, -t) - f_{r(n - 1)}(s, -t) - 1$ if and only if $f_{rn}(s, -t) - f_{r(n - 1)}(s, -t) - 1 \le \frac{1}{B(n)} < f_{rn}(s, -t) - f_{r(n - 1)}(s, -t)$. \\

The proof that $\frac{1}{B(n)} < f_{rn}(s, -t) - f_{r(n - 1)}(s, -t)$ is very similar to the corresponding proof in Theorem 2 for $f_n(s, t)$ (see \cite{ACMS}). It suffices to show that $1 < B(n)(f_{rn}(s, -t) - f_{r(n - 1)}(s, -t))$. Note that the first term of the product $B(n) f_{rn}(s, -t)$ is equal to 1, which we subtract from both sides of the inequality. Now we will compare the term for $k = n + m$ ($m \ge 1$) in $f_{rn}(s, -t) B(n)$ with the term for $k = n + m - 1$ in $f_{r(n - 1)}(s, -t) B(n)$ and see that it suffices to show that

\begin{align*}
0 < \frac{f_{rn}(s, -t)}{f_{r(n + m)}(s, -t)} - \frac{f_{r(n - 1)}(s, -t)}{f_{r(n + m - 1)}(s, -t)}. 
\end{align*}

By Lemma 1, we have that \[f_{rn}(s, -t) f_{r(n + m - 1)}(s, -t) - f_{r(n - 1)}(s, -t) f_{r(n + m)}(s, -t) = t^{r(n - 1)} f_r(s, -t) f_{rm}(s, -t).\] Dividing by $f_{r(n + m -1)}(s, -t) f_{r(n + m)}(s, -t)$, we have that \[\frac{f_{rn}(s, -t)}{f_{r(n + m)}(s, -t)} - \frac{f_{r(n - 1)}(s, -t)}{f_{r(n + m - 1)}(s, -t)} = t^{r(n - 1)} \frac{f_r(s, -t) f_{rm}(s, -t)}{f_{r(n + m)}(s, -t) f_{r(n + m - 1)}(s, -t)}\]. We claim that $f_k(s, -t) > 0$ for all $k \ge 1$. It suffices to show that the sequence $\{f_k(s, -t)\}$ is monotonically increasing since $f_1(s, -t) = 1$. This can be proved by induction. We have $f_1(s, -t) = 1$ and $f_2(s, -t) = s \ge 3$. Assume that $f_n \ge f_{n - 1}$. Then, $f_{n + 1} - f_n = (s - 1)f_n - tf_{n - 1} \ge (s - 1)f_n - (s - 1)f_{n - 1} = (s - 1)(f_n - f_{n - 1}) \ge 0$. Thus, $f_k(s, -t) > 0$ for all $k \ge 1$. This means that 

\begin{align*}
\frac{f_{rn}(s, -t)}{f_{r(n + m)}(s, -t)} - \frac{f_{r(n - 1)}(s, -t)}{f_{r(n + m - 1)}(s, -t)} = t^{r(n - 1)} \frac{f_r(s, -t) f_{rm}(s, -t)}{f_{r(n + m)}(s, -t) f_{r(n + m - 1)}(s, -t)} > 0. 
\end{align*} 

As for the other bound, the same procedure as the one used in the previous paragraph can be used to show that proving this bound reduces to showing that

\begin{align*}
\frac{f_{rn}(s, -t)}{f_{r(n + m)}(s, -t)} - \frac{f_{r(n - 1)}(s, -t)}{f_{r(n + m - 1)}(s, -t)} \le \frac{1}{f_{r(n + m - 1)}(s, -t)}.
\end{align*} 

\noindent Cross-multiplying and using Lemma 1, we find that it suffices to show that \[t^{r(n - 1)} \le \frac{1}{f_r(s, -t)} \frac{f_{r(n + m)}(s, -t)}{f_{rm}(s, -t)}\]. In \cite{ACMS}, it was shown that \[t^{r(n - 1)} \le \frac{1}{t} \frac{1}{f_r(s, t)} \frac{f_{r(n + m)}(s, t)}{f_{rm}(s, t)} \le \frac{1}{t} \frac{1}{f_r(s, -t)} \frac{f_{r(n + m)}(s, t)}{f_{rm}(s, t)}\]. Since $\frac{1}{f_r(s, -t)} \ge \frac{1}{f_r(s, t)}$ for all $r, s, t \in \mathbb{Z}^+$, it suffices to show that \[\frac{1}{t} \frac{f_{r(n + m)}(s, t)}{f_{rm}(s, t)} \le \frac{f_{r(n + m)}(s, -t)}{f_{rm}(s, -t)}\] for sufficiently large $s, t$. Let $X_1 = \frac{s + \sqrt{s^2 - 4t}}{2}$, $Y_1 = \frac{s - \sqrt{s^2 - 4t}}{2}$, $X_2 = \frac{s + \sqrt{s^2 + 4t}}{2}$, and $Y_2 = \frac{s - \sqrt{s^2 + 4t}}{2}$. Rewriting the inequality in terms of the $X_i$ and $Y_i$, we claim that \[\frac{1}{t} \frac{X_2^{r(n + m)} - Y_2^{r(n + m)}}{X_2^{rm} - Y_2^{rm}} \le \frac{X_1^{r(n + m)} - Y_1^{r(n + m)}}{X_1^{rm} - Y_1^{rm}}\] for sufficiently large $s, t$. We have \[\frac{1}{t} \frac{X_2^{r(n + m)} - Y_2^{r(n + m)}}{X_2^{rm} - Y_2^{rm}} \le \frac{X_1^{r(n + m)} - Y_1^{r(n + m)}}{X_1^{rm} - Y_1^{rm}}\] if and only if \[\frac{(X_2^{r(n + m)} - Y_2^{r(n + m)})(X_1^{rm} - Y_1^{rm})}{(X_1^{r(n + m)} - Y_1^{r(n + m)})(X_2^{rm} - Y_2^{rm})} \le t\]. Note that $\lim_{s \to \infty} \frac{X_1}{X_2} = 1$ and $\lim_{s \to \infty} \frac{Y_1}{Y_2} = -1$. We obtain the first limit as follows. Since $s > t$, we have $\frac{s + \sqrt{s^2 - 4s}}{s + \sqrt{s^2 + 4s}} < \frac{X_1}{X_2} < 1$. Since $\lim_{s \to \infty} \frac{s + \sqrt{s^2 - 4s}}{s + \sqrt{s^2 + 4s}} = 1$, we have  $\lim_{s \to \infty} \frac{s + \sqrt{s^2 - 4t}}{s + \sqrt{s^2 + 4t}} = 1$ for all $t < s$. We also have that $\lim_{s \to \infty} \frac{Y_i}{X_i} = 0$. Rewriting \[\frac{(X_2^{r(n + m)} - Y_2^{r(n + m)})(X_1^{rm} - Y_1^{rm})}{(X_1^{r(n + m)} - Y_1^{r(n + m)})(X_2^{rm} - Y_2^{rm})}\] in terms of $\frac{X_2}{X_1}$, $\frac{Y_2}{Y_1}$, and $\frac{Y_i}{X_i}$, we have the following.

\begin{align*}
\frac{(X_2^{r(n + m)} - Y_2^{r(n + m)})(X_1^{rm} - Y_1^{rm})}{(X_1^{r(n + m)} - Y_1^{r(n + m)})(X_2^{rm} - Y_2^{rm})} \\= \frac{\left(1 - \left(\frac{Y_2}{X_2}\right)^{r(n + m)}\right) \left(1 - \left(\frac{Y_1}{X_1}\right)^{rm} \right)}{\left(\left(\frac{X_1}{X_2}\right)^{r(n + m)} - \left(\frac{Y_2}{X_2}\right)^{r(n + m)} \left(\frac{Y_1}{Y_2}\right)^{r(n + m)}\right) \left(\left(\frac{X_2}{X_1}\right)^{rm} - \left(\frac{Y_1}{X_1}\right)^{rm} \left(\frac{Y_2}{Y_1}\right)^{rm}\right)}
\end{align*}

Using the limits given above, we find that \[\lim_{s \to \infty} \frac{(X_2^{r(n + m)} - Y_2^{r(n + m)})(X_1^{rm} - Y_1^{rm})}{(X_1^{r(n + m)} - Y_1^{r(n + m)})(X_2^{rm} - Y_2^{rm})} = 1\]. Thus, the inequality \[\frac{(X_2^{r(n + m)} - Y_2^{r(n + m)})(X_1^{rm} - Y_1^{rm})}{(X_1^{r(n + m)} - Y_1^{r(n + m)})(X_2^{rm} - Y_2^{rm})} \le t\] holds for any $t \ge 2$ for sufficiently large $s$. \\


Now we turn to the sum $\sum_{k = n}^\infty \frac{1}{f_{rk}(s, -1)^2}$. Let $C(n) = \sum_{k = n}^\infty \frac{1}{f_{rk}(s, -1)^2}$. Note that \\ $\left\lfloor \left(\sum_{k = n}^\infty \frac{1}{f_{rk}(s, -t)^2}\right)^{-1} \right\rfloor = f_{rn}(s, -t)^2 - f_{r(n - 1)}(s, -t)^2 - 1$ if and only if \[f_{rn}(s, -t)^2 - f_{r(n - 1)}(s, -t)^2 - 1 \le \frac{1}{C(n)} < f_{rn}(s, -t)^2 - f_{r(n - 1)}(s, -t)^2\]. \\

The proof that $\frac{1}{C(n)} < f_{rn}(s, -1)^2 - f_{r(n - 1)}(s, -1)^2$ is also very similar to the corresponding proof of Theorem 2. It suffices to show that $1 < C(n)(f_{rn}(s, -1)^2 - f_{r(n - 1)}(s, -1)^2)$. Note that the first term of the product $C(n) f_{rn}(s, -1)^2$ is equal to 1, which we subtract from both sides of the inequality. Now we will compare the term for $k = n + m$ ($m \ge 1$) in $f_{rn}(s, -t)^2 C(n)$ with the term for $k = n + m - 1$ in $f_{r(n - 1)}(s, -t)^2 C(n)$ and see that it suffices to show that

\begin{align*}
0 < \left(\frac{f_{rn}(s, -1)}{f_{r(n + m)}(s, -1)}\right)^2 - \left(\frac{f_{r(n - 1)}(s, -1)}{f_{r(n + m - 1)}(s, -1)}\right)^2. 
\end{align*}

From the proof of the first part of the theorem, we have $\frac{f_{rn}(s, -t)}{f_{r(n + m)}(s, -t)} > \frac{f_{r(n - 1)}(s, -t)}{f_{r(n + m - 1)}(s, -t)}$ and we can obtain the above inequality by squaring both sides and setting $t = 1$. \\

Consider the difference between the term for $k = n + m$ in $f_{rn}(s, -1)^2 C(n)$ and the term for $k = n + m - 1$ in $f_{r(n - 1)}(s, -1)^2 C(n)$. After replacing each fraction in the inequality \[\frac{f_{rn}(s, -1)}{f_{r(n + m)}(s, -1)} - \frac{f_{r(n - 1)}(s, -1)}{f_{r(n + m - 1)}(s, -1)} \le \frac{1}{f_{r(n + m - 1)}(s, -1)}\] with its square and clearing denominators, we have 

\begin{align*}
f_{rn}(s, -1)^2 f_{r(n + m - 1)}(s, -1)^2 - f_{r(n - 1)}(s, -1)^2 f_{r(n + m)}(s, -1)^2 \le f_{rn + rm}(s, -1)^2.
\end{align*}

If this inequality is satisfied for all $m \in \mathbb{Z}^+$, then the difference between the term for $k = n + m$ in $f_{rn}(s, -1)^2 C(n)$ and the term for $k = n + m - 1$ in $f_{r(n - 1)}(s, -1)^2 C(n)$ is always positive. If it is not satisfied for any $m \in \mathbb{Z}^+$, this difference is always negative. Let $X = \frac{s + \sqrt{s^2 - 4}}{2}$ and $Y = \frac{s - \sqrt{s^2 - 4}}{2}$. Note that $\lim_{s \to \infty} \frac{Y}{X} = 0$. \\

By Lemma 1, we have that \[f_{rn}(s, -1)^2 f_{r(n + m - 1)}(s, -1)^2 - f_{r(n - 1)}(s, -1)^2 f_{r(n + m)}(s, -1)^2 \le f_{rn + rm}(s, -1)^2\] if and only if \[f_r(s, -1) f_{rm}(s, -1) (f_{rn}(s, -1) f_{r(n + m - 1)}(s, -1) + f_{r(n - 1)}(s, -1) f_{r(n + m)}(s, -1))  \le f_{rn + rm}(s, -1)^2\]. This holds if and only if

\begin{align*}
\frac{1}{f_{rn}(s, -1) f_{r(n + m - 1)}(s, -1) + f_{r(n - 1)}(s, -1) f_{r(n + m)}(s, -1)} \frac{f_{rn + rm}(s, -1)}{f_r(s, -1)} \frac{f_{rn + rm}(s, -1)}{f_{rm}(s, -1)} \ge 1.
\end{align*}

We have that $f_n(s, -1) = \frac{X^n - Y^n}{X - Y}$ \cite{ACMS}. Since $\lim_{s \to \infty} \frac{Y}{X} = 0$, we have $\lim_{s \to \infty} f_n(s, -1) = X^{n - 1}$. This means that

\begin{flalign*}
\lim_{s \to \infty} \frac{1}{f_{rn}(s, -1) f_{r(n + m - 1)}(s, -1) + f_{r(n - 1)}(s, -1) f_{r(n + m)}(s, -1)} \frac{f_{rn + rm}(s, -1)}{f_r(s, -1)} \frac{f_{rn + rm}(s, -1)}{f_{rm}(s, -1)} \\= \frac{X^{2rn + 2rm - r - rm}}{X^{rn - 1} X^{r(n + m - 1) - 1} + X^{r(n - 1) - 1} X^{r(n + m) - 1}} \\
= \text{        } \frac{X^{2rn + 2rm - r - rm}}{X^{rn - 1 + rn + rm - r - 1} + X^{rn - r - 1 + rn + rm - 1}} \\
= \text{     } \frac{X^{2rn + rm - r}}{2X^{2rn + rm - r - 2}} \\
= \frac{X^2}{2}.
\end{flalign*}     

Since $X^2 > 2$ for sufficiently large $s$, we have

\begin{align*}
\frac{1}{f_r(s, -1) f_m(s, -1) (f_{rn}(s, -1) + f_{r(n + m - 1)}(s, -1))} \frac{f_{rn + rm}(s, -1)}{f_r(s, -1)} \frac{f_{rn + rm}(s, -1)}{f_m(s, -1)} > 1
\end{align*}

for sufficiently large $s$ for all $m \in \mathbb{Z}^+$. 



\end{proof}

\section{Periodicity modulo $n$}

Returning to periodicity of generalized Fibonacci sequences modulo $n$, we find a way to interpret the rank of a sequence mod $p$ as the order of an element of the splitting field of the characteristic polynomial of the recurrence in the case where $t = -1$. We also look at possible generalizations for other $t$ where $p \nmid t$ and the rank mod $p^k$. 

\begin{thm}
Let $X = \frac{s + \sqrt{s^2 - 4}}{2}$, $Y = \frac{s - \sqrt{s^2 - 4}}{2}$, and $D = s^2 -4$ be the discriminant of the characteristic polynomial of the recurrence with $D \not\equiv 0$ (mod $p$). If $(D/p) = 1$, then $r(p) = \frac{1}{2} ord(X)$ in $\mathbb {F}_p^{\times}$ if $ord(X)$ is even and $r(p) = ord(X)$ if $ord(X)$ is odd. Otherwise, we have that $r(p) = \frac{1}{2} ord(X)$ in $\mathbb {F}_{p^2}^{\times}$ if $ord(X)$ is even and $r(p) = ord(X)$ if $ord(X)$ is odd.
\end{thm}

\begin{proof}
Consider the matrix for the recurrence $U = \begin{pmatrix} s&t\\ 1&0 \end{pmatrix}$. Note that multiplying $U$ by $\begin{pmatrix} f_n(s, -1) \\ f_{n + 1}(s, -1) \end{pmatrix}$ gives $\begin{pmatrix} f_{n + 1}(s, -1) \\ f_{n + 2}(s, -1) \end{pmatrix}$. If $(D/p) = 1$, we can consider the eigenvalues $X$ and $Y$ in $\mathbb{F}_p$. Since $D \not\equiv 0$ (mod $p$), the eigenvalues of the matrix $U$ are distinct and $U$ is diagonalizable. Thus, we can write $U = C \begin{pmatrix} X & 0 \\ 0 & Y \end{pmatrix} C^{-1}$ or $U = \begin{pmatrix} Y & 0 \\ 0 & X \end{pmatrix}$, where $C$ is an invertible matrix. This means that the rank is the smallest positive integer $h$ such that $U^h = \lambda I$ for some $\lambda \in \mathbb{F}_p$. In terms of the eigenvalues, it is the smallest positive integer $h$ such that $X^h = Y^h$ in $\mathbb{F}_p$. Since $X$ and $Y$ are roots of $x^2 - sx + 1$, we have that $XY = 1$. So, we can rewrite $X^h = Y^h$ as $X^h = X^{-h}$ and get $X^{2h} = 1$ in $\mathbb{F}_p$. So, $r(p) = \frac{1}{2} ord(X)$ if $ord(X)$ is even and $r(p) = ord(X)$ if $ord(X)$ is odd. \\

We can use a similar argument for the case where $(D/p) = -1$. However, $X$ and $Y$ cannot be considered in $\mathbb{F}_p$, so we look at the splitting field $\mathbb{F}_{p^2}$ of the characteristic polynomial instead. Note that $X$ and $Y$ are also distinct in $\mathbb{F}_{p^2}$ since $X^p = Y$ in $F_{p^2}$ and $Y \notin \mathbb{F}_p$. Using the same steps as above, we find that $r(p) = \frac{1}{2} ord(X)$ in $\mathbb F_{p^2}^{\times}$ if $ord(X)$ is even and $r(p) = ord(X)$ if $ord(X)$ is odd.
\end{proof}

\begin{rem}
The use of the splitting field of the characteristic polynomial in the proof of this theorem is similar to its use in \cite{GRS} to study periods of generalized Fibonacci sequences.
\end{rem}

\begin{rem}
The rank of generalized Fibonacci sequences is a special case of the restricted period of a general linear recurrence (see \cite{EPSW} for a definition). 
\end{rem} 

We now look at the distribution of $r(p)$ over different values of $s$. \\

\begin{center}
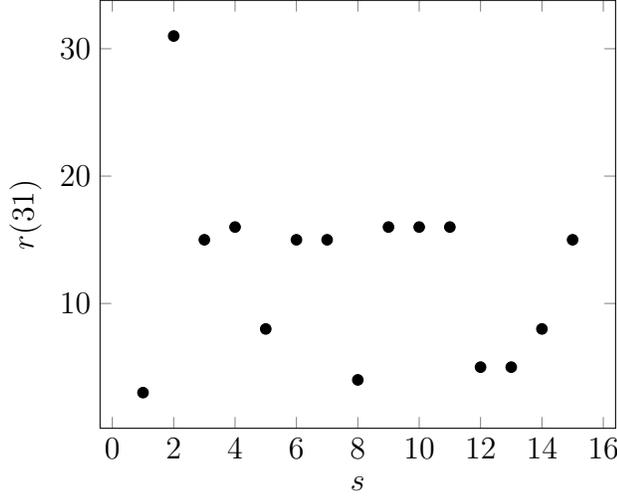
\begin{figure}
\caption{A plot of $r(31)$ for various values of $s$ when $t = -1$. Since the sequence starts with $0, 1, s, \ldots$, the terms of the sequence mod $p$ depend only on the residue of $s$ mod $p$. In addition, only $\frac{p - 1}{2}$ of the residues need to be considered since the terms of $f_n(-s, 1)$ mod $p$ are either identical or of opposite sign, which does not change the rank of the sequence.}
\begin{tikzpicture}
		\tikzstyle{every node} = [font = \large]
	\begin{axis}[xlabel={$s$}, ylabel={$r(31)$},%
	clickable coords={(xy): \thisrow{label}},%
	scatter/classes={%
		a={mark=*,black},%
		b={mark=triangle*,red},%
		c={mark=o,draw=black}}]
	\addplot[scatter,only marks,%
		scatter src=explicit symbolic]%
	table[meta=label] {
x     y      label
1			3				a
2			31			a
3			15			a
4			16			a	
5			8				a
6			15			a
7			15			a
8			4				a
9			16			a
10		16			a
11		16			a
12		5				a	
13		5				a
14		8				a
15		15			a

	};
	\end{axis}
\end{tikzpicture}
\end{figure}
\pagebreak
\end{center}


As we can see above for $p = 31$, there is a large amount of clustering around $\frac{p \pm 1}{2}$ for primes $p$. This can be understood using the distribution of orders of elements of finite abelian groups since $r(p) = \frac{1}{2} ord(X)$ or $r(p) = ord(X)$ in $\mathbb{F}_p$ or $\mathbb{F}_{p^2}$. In the case where $(D/p) = 1$, we are looking at the orders of elements of the cyclic group $\mathbb{F}_p^{\times}$. Take a generator $x$ of $\mathbb{F}_p^{\times}$ and a divisor $d$ of $p - 1$. Then, $x^k$ is of order $\frac{p - 1}{\text{gcd($p - 1$, $k$)}}$. So, there are $\varphi(d)$ elements of order $d$. Since $\varphi(m) \le \varphi(n)$ for all $m | n$, this means that there are generally more elements which are of higher order than lower order although $\varphi(n)$ does not increase monotonically. Since we have that $X^{p - 1} = 1$ in $\mathbb{F}_p$, the order of $X$ divides $p - 1$ and the previous statement applies. In the case that $ord(X)$ is odd, $ord(X) \le \frac{p - 1}{2}$ and we can repeat the observations which we made earlier. \\ 

If $(D/p) = -1$, we can consider the order of $X$ in $\mathbb{F}_{p^2} = \{a + bX: a, b \in \mathbb{F}_p\}$. Since $X^p = Y$ in $\mathbb{F}_{p^2}$ and $XY = 1$, we have that $X^{p + 1} = 1$ in $\mathbb{F}_{p^2}$ and $ord(X) | p + 1$. In addition, $\mathbb{F}_{p^2}^{\times}$ is cyclic since multiplicative subgroups of finite fields are cyclic. As above, we have that the number of elements of order $d$ is $\varphi(d)$ and $\varphi(a) \le \varphi(b)$ for all $a | b$. This means that the largest number of elements have order $p + 1$ among elements whose orders divide $p + 1$. \\

Making more precise statements about the rank mod $p$ would involve looking at the order of an element in the multiplicative groups for the fields $\mathbb{F}_p$ and $\mathbb{F}_{p^2}$ more closely. In addition, we have yet to determine a relationship between $s$ and the rank. \\

Since the rank mod $p^k$ exists if and only if $p \nmid t$, it is possible to generalize the order arguments above for other $t$ not divisible by $p$. The more interesting generalization relates to rank modulo general prime powers $p^k$. In this case we would consider the order of $X$ in $\mathbb{Z}_{p^k}^{\times}$ or in $\mathbb{Z}_{p^k}[X]^{\times}$. Finding the order would be more complicated than for $r(p)$ since $\mathbb{Z}_{p^k}$ is not a field for $k > 1$. However, we can still find an upper bound for $r(p^k)$ using a counting argument. \\

\begin{thm}
For any $s, t \in \mathbb{Z}$ with $p \nmid t$, we have that $r(p^k) \le p^k + 1$.
\end{thm}

\begin{proof}
Note that we can divide the set of possible pairs of terms modulo $p$ into equivalence classes where two pairs $(a, b)$ and $(c, d)$ belong to the same equivalence class if and only if $(c, d) \equiv (ka, kb)$ (mod $p$) for some $k$ such that $p \nmid k$. We claim that all pairs of consecutive terms $(f_n, f_{n + 1})$ with $n < r(p^k)$ belong to different equivalence classes. Assume that this is not the case. Then, there are two pairs $(a, b)$ and $(ka, kb)$ when considered modulo $p$ with $k \not\equiv 0,1$ (mod $p^k$). Since none of the terms of the sequence between $(a, b)$ and $(ka, kb)$ are divisible by $p^k$ and the terms following $ka$ are those following $a$ multiplied by $k$, there are no nonzero terms of the sequences considered modulo $p^k$ other than $f_0(s, t)$. However, this is a contradiction since $r(p^k)$ exists. This means that there are no such pairs $(a, b)$ and $(ka, kb)$ before $f_{r(p^k)}$. Note that sets of pairs of the form $(ka, kb)$ form an equivalence class. Since a single equivalence class contains $p^k - 1$ elements and each term of the sequence before $f_m$ must come from distinct equivalence classes, we have that $r(p^k) \le \frac{(p^k)^2 -1}{p^k - 1} = p^k + 1$. 
\end{proof}

Note that it suffices to look at $r(p^k)$ in order to understand $r(n)$ for any $n \in \mathbb{N}$ by the following theorem, which we can obtain by considering the matrix $U$ in the proof of Theorem 19. 

\begin{thm} [Robinson] \cite{Ro}
Let $m_1, m_2$ be positive integers greater than or equal to 2. Then we have that and $r([m_1, m_2]) = [r(m_1), r(m_2)]$.
\end{thm}

It would be interesting to find a relationship between $r(p^k)$ and $s$ for a given value of $t$ and adapt methods used above for general linear recurrences.

\section{Acknowledgements}
This research was conducted at the University of Minnesota Duluth REU program, supported by NSF/DMS grant 1062709 and NSA grant H98230-11-1-0224. I would like to thank Joe Gallian for his encouragement and creating such a great environment for research at UMD. I would also like to thank Krishanu Sankar and Sam Elder for their help with my research. I would especially like to thank Brian Chung for very helpful discussions at various points of this project.

\end{document}